\documentclass[11pt]{article}

\usepackage{amsthm,amsmath,amssymb,amsbsy,amsfonts,graphics,epsfig}
\newcounter{geqncount} %
{\setcounter{equation}{\value{geqncount}}}

\newtheorem{theo}{Theorem}[section]
\newtheorem{coro}{Corollary}[section]
\newtheorem{lemma}{Lemma}[section]

\newcommand{\beq}{\begin{equation}} \newcommand{\eeq}{\end{equation}}
\newcommand{\beqr}{\begin{eqnarray}}
\newcommand{\eeqr}{\end{eqnarray}}
\newcommand{\beqrn}{\begin{eqnarray*}}
\newcommand{\eeqrn}{\end{eqnarray*}} \newcommand{\brr}{\begin{array}}
\newcommand{\err}{\end{array}}
\newcommand{\bseq}{\begin{subequations}}
\newcommand{\eseq}{\end{subequations}}
\newcommand{\bef}{\begin{figure}} \newcommand{\eef}{\end{figure}}
\newcommand{\bec}{\begin{center}} \newcommand{\eec}{\end{center}}

\newcommand{\bpm}{\begin{pmatrix}}
\newcommand{\epm}{\end{pmatrix}}

\title{A note on the eigenvalues of double band matrices }

\author{Tyler McMillen\footnote{
Department of Mathematics,
California State University at Fullerton, Fullerton, CA  92834, USA (tmcmillen@fullerton.edu).}
}

\begin{document}

\maketitle

\bibliographystyle{plain}
%\bibliographystyle{mrl}

%\hrule

\begin{abstract}

\noindent 
We consider matrices containing two diagonal bands of positive entries.  We show that all eigenvalues of such matrices are of the form $r\,\zeta$, where $r$ is a nonnegative real number and $\zeta$ is a $p$th root of unity, where $p$ is the period of the matrix, which is computed from the distance between the bands.

%\vskip 10pt \hrule \vskip 10pt

\end{abstract}

%\section{Introduction}

\section{Introduction}
\label{s1}

Let $b$ and $k$ be positive integers and consider the double band matrix
\begin{equation}
A=\bpm 0 & 0 & \cdots & a_{1,k+1} & 0 & \cdots & 0 \\
0 & 0 & \cdots & 0 & a_{2,k+2}  & \cdots & 0\\
\vdots & & & &&\ddots  &\\
a_{b+1,1} & 0 & \cdots & &&& \\
0&a_{b+2,2} &   \cdots & &&&  \\
&& \ddots &&&& 
\epm ,
\label{A}
\end{equation}
where the diagonals $a_{b+j,j}, a_{i,k+i}$ are positive real numbers and the remaining entries are zero.  Such matrices arise, for example, in second order differential equations.  See   \cite{mcboag08} for an application to the Lam\'e equation.
The matrices as in \eqref{A} have a period $p$\footnote{The period is also sometimes referred to as the \textit{index}, or \textit{index of imprimitivity}.  A primitive matrix is one in which $p=1$.}, which is the greatest common divisor of the lengths of cycles in the directed graph $\Gamma(A)$.  Thus it is given by 
\beq 
p=\frac{b+k}{\mbox{gcd}(b,k)}.
\eeq
If $b$ and $k$ are relatively prime then $A$ is irreducible.  The
Perron-Frobenius theorems \cite{hojo85} tells us that the spectrum of a nonnegative and irreducible matrix with period $p$ is invariant under rotations by $2\pi/p$ and that there are exactly $p$ eigenvalues with maximal modulus $\rho(A)$ given by $\rho(A) \exp(i2\pi j/p),\; j=0,\dots,p-1$.  The main result of this paper is to show that for matrices with the special form  \eqref{A}, in addition to this, the eigenvalues all lie on the lines through the $p$th roots of unity in the complex plane.  
\begin{theo}
Let $A$ be an $n\times n$ matrix as in \eqref{A}.  Let $n=mp+q$ for $0\leq q<p$, and let $g=\mbox{gcd}(b,k)$.   If $g=1$ then $A$ has a zero eigenvalue of multiplicity $q$, and $mp$ eigenvalues
\begin{equation} 
r_s\, e^{i2\pi  j/p}, \;\; j = 0,\dots, p-1, \; s = 1, \dots, m,
\end{equation}
where $r_s,\; s=1,\dots, m$, are distinct, real and positive.  

More generally, the nonzero eigenvalues of $A$ are given by
\beq
r_{s,t}\, e^{i2\pi j/p}, \;\; j=0, \dots, p-1, \;\; t = 1,\dots, g, \;\; s = 1,\dots, 
\left\lfloor \frac{ \left\lfloor \frac{n-t}{g}\right\rfloor+1}{p} \right\rfloor ,
\label{g1}
\eeq
where for each fixed $t$, the $r_{s,t}$'s are distinct, real and positive.  
The zero eigenvalue of $A$ has multiplicity
\beq
g\left(\left\lfloor \frac{n}{g} \right\rfloor \mbox{ mod } p\right) + 
\left(n\mbox{ mod } g\right)\left[\left(\left\lfloor \frac{n}{g} \right\rfloor + 1\right) \mbox{ mod } p
- \left\lfloor \frac{n}{g} \right\rfloor \mbox{ mod } p \right]
\label{g2}
\eeq
\label{theo1}
\end{theo}
Before proceeding to the proof we make a few remarks.  First, if $b=k$, i.e. the bands are an equal distance from the diagonal, then $p=2$, so all the eigenvalues are real, and come in positive and negative pairs.  The case $b=1$, $k=2$ was considered in \cite{mcboag08}, where it was shown that all eigenvalues are of the form $r\exp(i2\pi  j/3)$.  Theorem \ref{theo1} is a generalization of the result in that paper.  The following corollary is an immediate consequence of the construction of the proof of the above theorem, and generalizes to the case when the two bands contain nonnegative and complex entries.

\begin{coro}
Let $A$ be an $n\times n$ complex matrix with $a_{ij}=0$ for all $i-j \neq b$ or $-k$.  Then the spectrum of $A$ is invariant under rotations by $2\pi/p$.  If, in addition, $a_{ij}$ is real and nonnegative, then the eigenvalues of $A$ are all of the form $r\,e^{i2\pi j/p}$, where $r$ is a nonnegative real number.
\label{cor1}
\end{coro}

\section{Proof of Theorem \ref{theo1}}

We first prove the following lemma that allows us to reduce the problem to the case when $b$ and $k$ are relatively prime.  If $g>1$ we can move the diagonals ``inward'' at the expense of rearranging them and inserting zeros.  In the proofs that follow we will employ many of the results from the theory of nonnegative matrices, which can be found in \cite{mi88} and \cite{hojo85}.

\begin{lemma}
$A$ is cogredient\footnote{A matrix $A$ is cogredient to $B$ if there is a permutation matrix $P$ such that $A=PBP^T$.} to a direct sum of $g$ matrices of the form \eqref{A} where the $b$ and $k$ are relatively prime.  That is, there is a permutation matrix $P$ such that
\begin{equation}
B = PAP^T = \bpm B_1 & 0 & \cdots & 0 \\ 0 & B_2 & \cdots & 0 \\  &  & \ddots &  \\ 0 & 0 & \cdots & B_g\epm
\label{B}
\end{equation}
where $B_i$ is an $\left( \left\lfloor \frac{n-i}{g}\right\rfloor + 1\right) \times \left( \left\lfloor \frac{n-i}{g}\right\rfloor + 1\right)$ square matrix with non-zero entries at the positions $(n_b+j,j)$ and $(i,n_k+i)$ and zeros everywhere else, where $n_b = b/g$ and $n_k = k/g$.
\end{lemma}
\begin{proof}
We construct $P$ explicitly.  Let $\sigma$ be the permutation of the integers $1,\dots,n$ defined as follows.  For $i = 1,\dots, g$, define
\beq 
n_i = \left\lfloor \frac{n-i}{g}\right\rfloor + 1
\eeq
and the partial sums $N_1=0$, $N_i =  n_1 + \cdots n_{i-1}$ for $i=2,\dots g$.  A simple calculation shows that $n_1+n_2+\cdots + n_g = n$.  Now, for each $i=1,\dots, g$ define
\beq
\sigma_{N_i+j} = i + (j-1) g, \;\; j = 1, 2, \dots, n_i
\eeq
%\begin{equation}
%\Sigma_i = \left( i, i+g, i+2g, \dots, i+\left\lfloor \frac{n-i}{g}\right\rfloor g\right)
%\end{equation}
%Then $\sigma$ is constructed by concatenating these sequences.  This can be done in any order, but for simplicity we take
%\begin{equation}
%\sigma = \Sigma_1 \Sigma_2 \cdots \Sigma_g
%\end{equation}
Then $\sigma$ is a permuation of $(1,\dots, n)$, and we have 
\begin{eqnarray}
\sigma_{n_b+j} - \sigma_{j} & = & b, \\
\sigma_{n_k+i} - \sigma_i & = & k,
\end{eqnarray}
for all pairs $N_i<n_b+j, j\leq N_i + n_i$ for some $i$, and likewise for $n_k+i$, $i$.  Now we define the permutation matrix $P$ by setting $p_{j,\sigma_j} = 1$ for $j=1,\dots,n$.  Then, define $B=PAP^T$, so that the $(i,j)$th entry of $B$ is $b_{ij} = a_{\sigma_i,\sigma_j}$.  Since $a_{ij}\neq 0$ if and only if $i-j=b$ or $i-j = -k$, as long as $N_i<n_b+j, j\leq N_i + n_i$ for some $i$ and $N_l < n_k+i,i \leq N_l + n_l$ for some $l$, we have
\begin{eqnarray}
b_{n_b+j,j} & = & a_{\sigma_{n_b+j},\sigma_j} \; = \; a_{s+b,s} ,     \\
b_{i,n_k+i} & = & a_{\sigma_i,\sigma_{n_k+i}} \; = \; a_{t,t+k},
\end{eqnarray}
for some $s$ and $t$.

Notice that each square block $N_l < i,j \leq N_l + n_l$ on the diagonal of $B$ contains $n_l - n_b$ nonzero elements at the positions $(n_b + j, j)$ and $n_l - n_k$ nonzero elements in the positions $(i, n_k + i)$, for a total of $n-gn_b = n - b$ elements in the positions $(n_b+j,j)$ and $n-gn_k=n-k$ elements in the positions $(i,n_k+i)$.  Thus, since $PAP^T$ moves each element of $A$ to an unique position, we see that $B$ has the form \eqref{B}, as required.
\end{proof}

We can now proceed to the proof of Theorem \ref{theo1}.
\begin{proof}[Proof of Theorem \ref{theo1}]
Suppose, first of all, that $b$ and $k$ are relatively prime.  Then $p=b+k$.  Also, we may assume that $b\leq k$ without loss of generality.
Since $A$ is irreducible with period $p$ it is cogredient to a matrix in superdiagonal block form.  That is, there is a permutation matrix $P$ such that
\begin{equation}
C = PAP^T = 
\bpm 0   & C_1 &   0    &  \cdots & 0 \\
          0   & 0     & C_2  &  \cdots & 0 \\
          \vdots&   &          & \ddots & \\
          0       &  0 &      0   &  \cdots   & C_{p-1} \\
          C_p  & 0 & 0 & \cdots & 0
\epm,
\label{C}
\end{equation}
where the diagonal zero blocks are square, but the blocks $C_j$ may be rectangular.
The key to the proof is that we can find a $P$ such that the matrices $C_j$ are bidiagonal.  We construct P explicitly.  We first construct a permutation $\sigma$ of the integers from 1 to $n$.  
For each $i=0,\dots, p-1$ define $\gamma_i, z_i$, and $n_i$ by
\begin{eqnarray} 
\gamma_i &=& (ik+1)\mbox{ mod } p = (ik+1) + z_i, \\
n_i & = & \left\lfloor \frac{n-\gamma_i}{p}\right\rfloor + 1,
\end{eqnarray}
and $n_{-1}=0$.  An elementary calculation shows that $n_0+\cdots+n_{p-1} = n$.  To simplify notation we also define the partial sums $N_i = n_0 + \cdots + n_i$.  Now, for each $i=0,\dots, p-1$, define
\begin{equation}
\sigma_{N_{i-1} +j} = \gamma_{i}+(j-1)p, \;\;\; j=1,\dots,n_{i}.
\end{equation} 
Thus $\sigma = \left(\sigma_1,\dots,\sigma_n\right)$ is a permutation of the numbers $(1,\dots,n)$.

Now we define the permutation matrix $p_{j,\sigma(j)} = 1$.  Then the $(i,j)$ entry of $C=PAP^T$ is $c_{ij} = a_{\sigma_i,\sigma_j}$.  For ordered subsets $\alpha, \beta$ of $\{1,\dots, n\}$, let $A(\alpha,\beta)$ be the submatrix with rows in $\alpha$ and columns in $\beta$.  We define the submatrices $C_i$ as follows: 
\beqr
C_i & = & C\left(\left\{N_{i-2}+1,\dots, N_{i-2} + n_{i-1}\right\}, \left\{N_{i-1}+1,\dots, N_{i-1}+n_{i}\right\}\right) , \;\;\; i = 1,\dots, p-1, \\
C_p & = & C\left(\left\{N_{p-2}+1, \dots, n\right\},\left\{1,\dots, n_0\right\}\right)
\eeqr   
Note that since $b$ and $k$ are relatively prime, and $b\leq k$, $z_i-z_{i+1}$ is either 0 or 1.  We now use this fact to show that each $C_i$ is bidiagonal.

Consider the main, lower and upper diagonals of $C_i$ for $i=1,\dots, p-1$.  Note that
\begin{eqnarray}
\sigma_{N_{i-1}+j}-\sigma_{N_{i}+j} &=& \gamma_i + (j-1)p - \left(\gamma_{i+1}+(j-1)p\right) \nonumber \\
&=&  -k + \left(z_i-z_{i+1}\right)p \nonumber \\
&=& \left\{\begin{array}{ll} -k & \mbox{ if } z_i - z_{i+1} = 0 \\ b & \mbox{ if } z_i - z_{i+1} = 1
\end{array}\right. \\[.1in]
\sigma_{N_{i-1}+j+1} - \sigma_{N_i+j} & = & \gamma_i+jp -\left(\gamma_{i+1}+(j-1)p\right) \nonumber \\
& = & b + \left(z_i-z_{i+1}\right)p  \nonumber \\
& = & \left\{\begin{array}{ll} b & \mbox{ if } z_i - z_{i+1} = 0 \\ b + p & \mbox{ if } z_i - z_{i+1} = 1
\end{array}\right. \\[.1in]
\sigma_{N_{i-1}+j} - \sigma_{N_i+j+1} & = & \gamma_i+(j-1)p -\left(\gamma_{i+1}+jp\right) \nonumber \\
& = &  -k +\left(z_i-z_{i+1}\right)p - p  \nonumber \\
& = & \left\{\begin{array}{ll} -k-p & \mbox{ if } z_i - z_{i+1} = 0 \\ -k & \mbox{ if } z_i - z_{i+1} = 1
\end{array}\right.
\eeqr
Now, since $a_{i,j}\neq 0$ if and only if $i-j = b$ or $i-j = -k$, we see that each $C_i, i=1,\dots, p-1$, is  upper  bidiagonal if $z_{i-1}-z_{i} = 1$ and lower bidiagonal if $z_{i-1}-z_i =0$, and that the entries in the two nonzero diagonals are all positive.

Next, consider $C_p$.  We have
\begin{eqnarray}
\sigma_{N_{p-2}+j}-\sigma_{j} &=& \gamma_{p-1} + (j-1)p - \left(1+(j-1)p\right) \nonumber \\
&=&  b  \\
\sigma_{N_{p-2}+j}-\sigma_{j+1} &=& \gamma_{p-1} + (j-1)p - \left(1+jp\right) \nonumber \\
&=&  -k
\eeqr
Thus $C_p$ is an upper bidiagonal $n_{p-1}\times n_0$ matrix.

Now we have $C$ in the form \eqref{C}.  Thus
\beq
C^p = \bpm D_1 & 0 & \cdots & 0 \\
                      0 & D_2 & \cdots & 0 \\
                       &  & \ddots & \\
                       0 & 0  & 0  & D_p
                       \epm ,
\label{Cp}
\eeq
where
\beq
D_1 = C_1C_{2}\cdots C_{p}, \;\; D_2 = C_2C_{3}\cdots C_{p}C_1, \;\; \dots, \;\; D_p = C_pC_1\cdots C_{p-1}
\eeq
It follows that the non-zero spectra of each of the products $D_j$
are all equal and that the non-zero spectra of $C$ consists of the $p$th roots of the eigenvalues of $D_j$.  Notice that each of the $D_j$ are square $n_{j-1}\times n_{j-1}$ matrices.  Furthermore, for each $j$, $n_j = m=\left\lfloor n/p\right\rfloor$ or $n_j=m+1$, and for at least one $j$, $n_j=m$.  Thus, if $m=0$ all the eigenvalues are zero.  So we assume that $m>0$.

So let $j$ be such that $n_{j-1}=m$, so that $D_j$ is $m\times m$.  We will show that the eigenvalues of $D_j$ are all real, positive and distinct.  For this it is sufficient to show that $D_j$ is oscillatory.  First, we recall a few definitions.  A matrix is totally nonnegative (positive) if all minors of all sizes are nonnegative (positive).  A matrix $X$ is oscillatory if there exists a positive integer $s$ such that $X^s$ is totally positive.  Oscillatory matrices have the remarkable property that their eigenvalues are distinct, real and positive\cite[cf. Ch. 6, Theorem 3.2]{mi88}.  Moreover, a matrix $X$ is oscillatory if and only if it is (a) totally nonnegative, (b) nonsingular, and (c) satisfies $x_{i,i+1}, x_{i+1,i} > 0$ for all $i$.  We will verify these three conditions for $D_j$.

Condition (c) follows from the fact that $z_0=z_1 = 0$, so $C_1$ is lower bidiagonal.  And since $C_p$ is upper bidiagonal, $D_j$ is the product of upper and lower bidiagonal matrices.  Thus, the super and sub-diagonal entries of $D_j$ are all positive.  Conditions (a) and (b) follow from the Cauchy-Binet identity.  For any pair of ordered subsets $\alpha, \beta \subset \{1,2,\dots, m\}$ of the same cardinality, the determinant of $D_j(\alpha,\beta)$ is
\beqr
\mbox{det}\left(D_j(\alpha,\beta)\right) & = & \mbox{det}\left(\left( C_j C_{j+1}\cdots C_{[[j+p-1]]} \right)(\alpha, \beta)\right) \nonumber \\ 
& = & \sum_{\theta_{j},\dots, \theta_{j+p-2}} \prod_{i=j}^{(j+p-1)} \mbox{det}\left(C_{[[i]]}\left(\theta_{i-1},\theta_i\right)\right)
\label{CB}
\eeqr
where $[[i]] = i\mbox{ mod } p$,  $\theta_{j-1} = \alpha, \theta_{j+p-1}=\beta$, and the sum is taken over all ordered subsets $\theta_j,\dots,\theta_{j+p-2} $ where the submatrices are defined.  A simple exercise shows that all minors of bidiagonal matrices with positive entries on the two diagonals are nonnegative.  It immediately follows from \eqref{CB} that $D_j$ is totally nonnegative.  Thus, condition (a) is satisfied.

It remains to be shown that $D_j$ is nonsingular.  For this we use \eqref{CB} again, with $\alpha=\beta=\{1,2,\dots, m\}$.  We have
\beqr
\mbox{det}\left(D_j\right) & = & \prod_{i=j}^{(j+p-1)} \mbox{det}\left(C_{[[i]]}\left(\alpha,\alpha\right)\right) \nonumber \\
&& + \sum_{\theta_{j},\dots, \theta_{j+p-2}\neq \alpha} \prod_{i=j}^{(j+p-1)} \mbox{det}\left(C_{[[i]]}\left(\theta_{i-1},\theta_i\right)\right),
\label{CB2}
\eeqr
where $\theta_{j-1} = \theta_{j+p-1}=\alpha$.  The leading principal submatrices $C_i(\alpha,\alpha)$ are bidiagonal with positive entries on their diagonals, so the first term in \eqref{CB2} is positive, and the remaining terms are nonnegative.  Hence $\mbox{det}\left(D_j\right) > 0$ and condition (b) is satisfied.

Since $D_j$ is oscillatory it has $m$ distinct, positive eigenvalues $\omega_1,\dots, \omega_m$, and thus the spectrum of $A$ consists of $q$ zeros together with the numbers
\beq 
\omega_s^{1/p} e^{i2\pi j/p}, \;\; j = 1,\dots, p, \; s = 1,\dots, m.
\eeq
This completes the proof in the case when $b$ and $k$ are relatively prime.

In the case when $g=\mbox{gcd}(b,k)>1$, we first form $B$ as the direct sum of matrices as in \eqref{B}.  The spectrum of $A$ is the union of the spectra of the $B_i$'s, so we apply the previous result to each of the matrices $B_1,\dots, B_g$.   Note that for $i=1,\dots, n\mbox{ mod }g$, $B_i$ is an 
$\left(\left\lfloor n/g\right\rfloor +1 \right) \times \left(\left\lfloor n/g\right\rfloor +1 \right) $ square matrix and for $i= n\mbox{ mod }g +1,\dots, g$, $B_i$ is an $\left\lfloor n/g\right\rfloor  \times \left\lfloor n/g\right\rfloor  $ square matrix.  The formulas \eqref{g1} and \eqref{g2} are obtained by applying the result for relatively prime $b$ and $k$ to each of the matrices $B_i$ and then counting.  In the general case, although all of the nonzero eigenvalues are of the form $r\,e^{i2\pi j/p}$, there is no guarantee that the $r$'s are all distinct since it may happen that the spectra of two or more of the $B_i$'s overlap.
\end{proof}

In the preceeding proof the form of $C=PAP^T$ in \eqref{C} and \eqref{Cp} depended only on the zero pattern of $A$.  So, as long as $a_{ij}=0$ outside of the two diagonal bands $(b+j,j)$ and $(i,k+i)$, the eigenvalues of $A$ are the $p$th roots of the eigenvalues of $D_j$, and hence the spectrum is invariant under rotations by $2\pi/p$.  Moreover, if we impose the weaker condition that $a_{b+j,j}, a_{i,k+i} \geq 0$, then all minors of the bidiagonal matrices $C_i$ are still nonnegative, so the $D_j$'s in \eqref{Cp} are all totally nonnegative.  The eigenvalues of a totally nonnegative matrix are real and nonnegative, but not necessarily distinct and positive \cite[cf. Ch. 6, Theorem 2.5]{mi88}.  Thus, we have proved Corollary \ref{cor1}.

%\vskip12pt \noindent {\bf Acknowledgements:}

%\bibliography{biblio}

\end{document}